\documentclass[smallextended,draft]{svjour3}
\usepackage[dvips]{graphicx}
\usepackage{latexsym}

\usepackage{amsmath,amssymb,amsfonts,amsthm}
\usepackage[utf8]{inputenc}
\usepackage{xcolor}
\usepackage{enumerate} 
\usepackage{xspace}

\usepackage{stmaryrd}
\usepackage{dsfont}
\usepackage{mathrsfs}
\usepackage{verbatim}

\newcounter{truc}

\newtheorem{Lemma2}[truc]{Lemma}
\newtheorem{Remarque}{\textnormal{\bf{R\scriptsize{EMARK}}}}
\newtheorem{Exemple}{\textnormal{\bf{E\scriptsize{XAMPLE}}}}

\newcommand \NN {\mathbb{N \xspace}}
\newcommand \RR {\mathbb{R \xspace}}
\newcommand \ZZ {\mathbb{Z \xspace}}
\newcommand \LL {\mathbb{L \xspace}}
\newcommand \EE {\mathbb{E \xspace}}
\newcommand \PP {\mathbb{P \xspace}}
\newcommand \II {\mathds{1 \xspace}}
\newcommand \VV {\mathbb{V \xspace}}
\renewcommand{\subclassname}{\textbf{Mathematics Subject Classification (2010)}  }

\author{Fabien Montégut}

\institute{Institut de Math\'ematiques de Toulouse UMR5219, Université de Toulouse CNRS, 118 Route de Narbonne, F 31062 Toulouse Cedex 9,    France. \email{fabien.montegut@math.univ-toulouse.fr}}

\title{Double asymptotic for random walks on hypercubes} 
\begin{document}

\maketitle

\begin{abstract} We consider the sum of the coordinates of a simple random walk on the $K$-dimensional hypercube, and prove a double asymptotic of this process, as both the time parameter $n$ and the space parameter $K$ tend to infinity. Depending on the asymptotic ratio of the two parameters, the rescaled processes converge towards either a "stationary Brownian motion", an Ornstein-Uhlenbeck process or a Gaussian white noise.
\keywords{Limit theorems \and Markov chains \and Hypercube}
\subclass{60F17 \and 60J10}
\end{abstract}

%

\section{Introduction}

Many results (like the Law of Large Numbers or the Central Limit Theorem) are already known for the asymptotic behavior in time of an additive functional of a Markov chain (see for instance \cite{MT09}). But the case where we consider a sequence of such processes is only partially studied. Here we adress the problem of a double asymptotic as both the time and the index in the sequence tend to infinity. For instance, a well understood case is the discretization of a diffusion process : as we consider larger time horizons and finer meshes, the discrete processes converge to the continuous diffusion they come from.

Actually this paper was initially motivated by the study of a constrained random walk introduced in \cite{BCEN15}, where an additive observable of a simple random walk on a graph $G_K$ whose vertices are $\lbrace -1,1\rbrace^K$ is described. The authors used a discrete Hodge decomposition to rewrite their observable as a sum of a divergence-free and a bounded gradient vector fields, and then proved that for every $K$ the rescaled constrained random walk converges in time to a Brownian motion with variance $\sigma_K^2=\frac{2}{K+2}$.

A natural generalization of this result would be to let $K$ grow to $+\infty$, but the diffusivity tends to $0$ as $K$ grows, which means that the normalization $\sqrt{n}$ used in \cite{BCEN15} is too strong to get a non-trivial limit in this case. Moreover the gradient part was neglected since it is bounded when $K$ is fixed. Actually, it is a function of $K$, and when $K$ tends to infinity it is no more obvious that it can be neglected.

In our setting we are dealing with a simplified version of this model. By removing some edges from the graphs $G_K$, the additive observable corresponds to a pure gradient term. Hence in this model we have $\sigma_K$ vanishing for every $K$. Moreover this toy model is more amenable to computations since the dependence in $K$ of the gradient term is quite simple. Even if the diffusivity is zero we managed to get a convergence to Gaussian processes when both $n$ and $K$ tend to $+\infty$.

Surprisingly we find out that the good normalization and the limiting process both depend on the asymptotic of the ratio of our parameters. Indeed, if the limit of $\frac{n}{K}$ is a positive constant, our rescaled process will converge to an Ornstein-Uhlenbeck process. If the ratio tends to $+\infty$ the limiting process is a Gaussian white noise (i.e. a collection of i.i.d. Gaussian random variables). Last but not least if $\frac{n}{K}$ tends to $0$ the initial value may diverge (for instance in the stationary case), hence we will prove first a  convergence to a Brownian motion if we subtract the value of the processes at time $0$ before considering a "stationary Brownian motion" (i.e. some weak form of Brownian motion starting from the Lebesgue measure) as a singular limit.

\section{The model}

Let the graph $H_K=(V_K,E_K)$ be the $K$-dimensional hypercube, more precisely :
\[V_K=\lbrace -1,+1\rbrace ^K \text{ and } E_K=\left\lbrace \lbrace u, u'_i\rbrace : u\in V_K, i\in [K]\right\rbrace\]
with $u'_i=(u^{(1)},\dots, u^{(i-1)},-u^{(i)},u^{(i+1)},\dots,u^{(K)})$ and $[K]\overset{def}{=}\llbracket 1,K\rrbracket$.

\medskip

We define $\left(Y_K(n)\right)_{n\geq 0}$ as the simple random walk on $H_K$ starting from a law $\mu_K$ on $V_K$. We set $f_K:V_K\rightarrow \RR$ the function giving the sum of the coordinates in $H_K$, namely :
\[\forall \ v\in V_K : \ f_K(v)=\sum_{i=1}^K v^{(i)}.\]

\bigskip

We are interested in the behavior of $f_K(Y_K(n))$, and more specifically we want to give some scaling limit as $n$ and $K$ both tend to infinity of the linear interpolations processes defined by :
\[X_{n,K}(t)=f_K(Y_K(\lfloor nt\rfloor)) \ \ t\geq 0.\]
In other words we want to find some $c_{n,K}$ and random process $\left( X_t\right)_{t\geq 0}$ such that :
\[\left( \frac{X_{n,K}(t)}{c_{n,K}}\right)_{t\geq 0} \underset{n,K\to\infty}{\longrightarrow} \left(X_t\right)_{t\geq 0}.\]
The mode of such convergences will be either the convergence in distribution in the set of càdlàg functions $D(\RR_+,\RR)$ (endowed with the Skorokhod topology used in \cite{EK91}), or a convergence of the finite-dimensional marginals (weakly or vaguely). In the next sections $X\overset{\mathcal{D}}{=}Y$ will mean that the random variables or processes $X$ and $Y$ both have the same probability law, and we will consider that $K$ is a function of $n$ (but we will keep writing $K$ instead of $K(n)$ to lighten the notations).

\bigskip

We begin with the intermediate regime (both parameters grow at comparable speeds) using diffusion approximation results from \cite{EK91}. Then in the fast regime ($n$ growing faster than $K$) the processes do no longer converge to a diffusion process, so we use an ersatz of Donsker's theorem (which will be proven in the appendix) to prove a convergence of the finite-dimensional laws. Lastly in the slow regime ($n$ grows slower than $K$) we state a convergence in law of the increments of the processes before proving a vague convergence to a Brownian motion starting from its invariant measure (i.e. the Lebesgue measure) using the results from the intermediate regime.

\bigskip

One can remark that $X_{n,K}$ is an affine transformation of an Ehrenfest's urn, which explains the convergence to an Ornstein-Uhlenbeck process in the intermediate regime (see for instance \cite{CM16} for a proof of this scaling limit). 

\section{Intermediate regime}

Consider the sequences of random variables defined by :
\begin{equation*}
Z_{n,K}(i) \overset{def}{=} \frac{f_K(Y_K(i))}{c_{n,K}}
\end{equation*}
for every $i\in\NN$. In this section we will assume that $K$ and $n$ grow at comparable speeds, namely there exists $\lambda >0$ such that :
\[\frac{n}{K}\underset{n\to\infty}{\longrightarrow}\lambda .\]

One can check that for every $n\geq 1$ the sequence $\left(Z_{n,K}(i)\right)_{i\in\NN}$ is an homogeneous Markov chain with values in :
\[\mathcal{S}_n=\left\lbrace \frac{2k-K}{c_{n,K}}:k\in\llbracket 0,K\rrbracket\right\rbrace\]
whose transition kernel is given for every $x\in\mathcal{S}_n$ by :
\[\PP_n(x,.)=\left(\frac{1}{2}+\frac{xc_{n,K}}{2K}\right)\delta_{x-\frac{2}{c_{n,K}}} + \left(\frac{1}{2}-\frac{xc_{n,K}}{2K}\right)\delta_{x+\frac{2}{c_{n,K}}}.\]

Using the same kind of proof as the Example 27.8 from \cite{CM16}, we simply have to compute the two following functions :
\begin{align*}
b_n(x) &\overset{def}{=}n\int_{\vert y-x\vert\leq 1}(y-x)\ \PP_n(x,dy) = -2x\frac{n}{K} \\
a_n(x) &\overset{def}{=}n\int_{\vert y-x\vert\leq 1}(y-x)^2\ \PP_n(x,dy) = 4\frac{n}{c_{n,K}^2}.
\end{align*}

We will mostly use the following result, which grants a convergence to diffusion processes for homogeneous Markov chains :
\begin{proposition}\label{EK}
Suppose there exist a random variable $X_0$ and two continuous functions $b:\RR\longrightarrow\RR$ and $a:\RR\longrightarrow\RR_+$ such that all the following properties hold for any $r,\varepsilon >0$ :
\begin{align*}
\sup_{\vert x\vert\leq r}\left\vert a_n(x)-a(x)\right\vert\underset{n\to\infty}{\longrightarrow}0,\\
\sup_{\vert x\vert\leq r}\left\vert b_n(x)-b(x)\right\vert\underset{n\to\infty}{\longrightarrow}0,\\
\sup_{\vert x\vert\leq r}n\times\PP_n(x,[x-\varepsilon,x+\varepsilon]^c)\underset{n\to\infty}{\longrightarrow}0,\\
\text{ and }\hspace{1cm} Z_{n,K}(0)\overset{\mathcal{D}}{\underset{n\to\infty}{\longrightarrow}}X_0.
\end{align*}
Then we get the following convergence of processes :
\begin{equation*}
\left(Z_{n,K}(\lfloor nt\rfloor)\right)_{t\geq 0}\overset{\mathcal{D}}{\underset{n\to\infty}{\longrightarrow}}\left(X_t\right)_{t\geq 0}
\end{equation*}
where $\left(X_t\right)_{t\geq 0}$ is the diffusion starting from $X_0$ and solving :
\[ dX_t=b(X_t)dt+\sqrt{a(X_t)}dW_t \]
with $(W_t)_{t\geq 0}$ a standard Brownian motion starting from $0$.
\end{proposition}
\begin{proof}
This result is just an adapted version of the Corollary 4.2 (p.355) from \cite{EK91} about the diffusion approximation.
\end{proof}

Once applied to the processes $Z_{n,K}$, it grants the following result :

\begin{theorem}\label{RegimInter}
Under the following assumptions :
\[\frac{n}{c_{n,K}^2}\underset{n\to\infty}{\longrightarrow}\sigma^2\geq 0, \hspace{1cm} \frac{n}{K}\underset{n\to\infty}{\longrightarrow}\lambda\geq 0 \hspace{0.5cm} \text{ and } \hspace{0.5cm} Z_{n,K}(0)\overset{\mathcal{D}}{\underset{n\to\infty}{\longrightarrow}} Z_0\]
we have the convergence in distributions of the random processes :
\[\left(Z_{n,K}(\lfloor nt\rfloor)\right)_{t\geq 0}\overset{\mathcal{D}}{\underset{n\to\infty}{\longrightarrow}}\left(O_{\lambda,\sigma} (t)\right)_{t\geq 0}\]
where $\left(O_{\lambda,\sigma}(t)\right)_{t\geq 0}$ denotes the diffusion process solving :
\[\left\lbrace \begin{array}{l} dO_{\lambda,\sigma}(t)=-2\lambda O_{\lambda,\sigma}(t)dt+2\sigma dW_t \\ O_{\lambda,\sigma}(0)\overset{\mathcal{D}}{=} Z_0. \end{array} \right. \]
\end{theorem}
\begin{proof}
We simply apply the Proposition \ref{EK} : $a_N(x)=4\frac{n}{c_{n,K}^2}$ converges uniformly to $a(x)=4\sigma^2$, $b_N(x)=-2x\frac{n}{K}$ converges locally uniformly to $b(x)=-2x\lambda$ and :
\[n\PP_n(x,[x-\varepsilon,x+\varepsilon]^c) = n\II_{\frac{2}{c_{n,K}}>\varepsilon}\]
which tends uniformly to 0 since the condition $\frac{n}{c_{n,K}^2}\underset{n\to\infty}{\longrightarrow}\sigma^2$ guaranties that $c_{n,K}$ tends to infinity.
\end{proof}

\begin{Remarque}
We allow $\lambda = 0$ in the Theorem \ref{RegimInter} because we will use this specific case later to prove the Theorem \ref{TheorRegimLent} in the slow regime. But keep in mind that the intermediate regime restrains to positive values of $\lambda$.
\end{Remarque}

We can distinguish two different cases for the intermediate regime in the previous theorem :

\begin{corollary}
\begin{itemize}
\item[$\bullet$]If $\sigma^2>0$ (i.e. $c_{n,K}$ is equivalent to $\sigma\sqrt{n}$) then the limit diffusion is an Ornstein-Uhlenbeck process. Moreover, if $Y_K(0)$ is uniformly distributed, $Z_{n,K}(0)$ converges in law to a $\mathcal{N}(0,\frac{\sigma^2}{\lambda})$ random variable, making the process $O_{\lambda, \sigma}$ temporally stationary.
\item[$\bullet$]If $\sigma=0$ (i.e. $c_{n,K}$ grows faster than $\sqrt{n}$) the limit process is the (random) function $Z_0e^{-2\lambda t}$. Note that if $c_{n,K}$ grows faster than $K$ or $Y_K(0)$ is uniform over $V_K$, then the limit process is constantly equal to $0$.
\end{itemize}
\end{corollary} 
\begin{Exemple}
If there exists $C\in [-1,1]$ such that $\frac{f_K(Y_K(0))}{K}\overset{\PP}{\underset{K\to\infty}{\longrightarrow}}C$, then we may set $c_{n,K}=K$ to check the assumptions of Theorem \ref{RegimInter} and the limit would be the deterministic function $t\mapsto Ce^{-2\lambda t}$.
\end{Exemple}

\section{Fast regime}

In this section we will consider regime $\frac{n}{K}\to\infty$, and we will always assume that $\mu_K$ is the uniform distribution on $V_K$. The Theorem \ref{RegimInter} let us think that the good normalization will be of the order of $\sqrt{K}$, but the limit process would no longer be a diffusion. Since we can't use the diffusion approximation theorems from \cite{EK91}, we need to take a closer look at the finite-dimensional laws of the processes. Let's define $\mathbf{t}_n=(t_1(n),\dots,t_s(n))$ such that
\begin{equation}\label{HypotTempsAleat}
\frac{n}{K}\left(t_{j+1}(n)-t_j(n)\right)\underset{n,K\to\infty}{\longrightarrow}+\infty \ \forall\  1\leq j\leq s-1
\end{equation} 
and set the following notation :
\[X_{n,K}(\mathbf{t}_n)=\left(X_{n,K}(t_1(n)),\dots,X_{n,K}(t_s(n)) \right).\]

Our goal is to show that :
\begin{equation}\label{EquatBut}
\frac{X_{n,K}(\mathbf{t}_n)}{\sqrt{K}}=\mathbf{F}\left(\left(\frac{1}{\sqrt{K}}\sum_{i\in B_k}\xi_i\right)_{k\in [N]}\right)
\end{equation}
for some linear functional $\mathbf{F}$, where $\left(\xi_i\right)_{i\in\NN}$ is an i.i.d. sequence of Rademacher random variables, and $\left(B_k\right)_{k\in [N]}$ is a random partition of $[K]$ (independent from the $\xi_i$).

\medskip

Since $X_{n,K}(t)=f_K(Y_K(\lfloor nt\rfloor))$, we will rewrite $f_K(Y_K(n))$ in a more suitable form :

\begin{proposition}\label{PropoChangNotat}
Let $\left(\xi_k\right)_{k\in\NN}$ be i.i.d. Rademacher random variables and let $\left(U_k^{(K)}\right)_{k\in\NN}$ be i.i.d. uniform random variables on $[K]$ independent from the $\xi_k$. Then :
\[f_K(Y_K(n))\overset{\mathcal{D}}{=}\sum_{i=1}^K\xi_i\left(\II_{i\notin O_0^n}- \II_{i\in O_0^n}\right)\]
where $O_i^j\overset{def}{=}\lbrace k\in [K] : \vert \lbrace i<l\leq j : U_l^{(K)}=k\rbrace\vert \text{ is odd}\rbrace$.
\end{proposition}
\begin{proof}
We use the simple fact that for all integer $n\geq 0$ :
\[f_K(Y_K(n+1))=f_K(Y_K(n))-2Y_K^{(U_n^{(K)})}(n).\]
By induction we get that :
\[f_K(Y_K(n))=f_K(Y_K(0))-2\sum_{i=1}^K  Y_K^{(i)}(0)\II_{\epsilon_i(n)\text{ odd}}= \sum_{i=1}^K (-1)^{\epsilon_i(n)} Y_K^{(i)}(0)\]
with $\epsilon_i(n)\overset{def}{=}\vert\lbrace U_k^{(K)}=i : k\in\llbracket 0,n-1\rrbracket\rbrace\vert$. Since only the oddness of $\epsilon_i(n)$ impacts the value of $f_K(Y_K(n))$, we get the expected result by denoting $\xi_i$ the value of $Y_K^{(i)}(0)$.
\end{proof}

Now that we made the $\left( \xi_i\right)_{i\in\NN}$ appear in the value of $X_{n,K}(\mathbf{t}_n)$, we want to construct a suitable partition to get the equation \eqref{EquatBut}.

Since we are only interested in the coordinates which have been drawn an odd number of times between $\lfloor nt_{i-1}(n)\rfloor$ and $\lfloor nt_i(n)\rfloor$, we define the following sets :
\begin{equation}\label{DefinB(J)}
\forall \ J\subset [s], \ B(J)\overset{def}{=}\left(\bigcap_{k\in J}O_{\lfloor nt_{k-1}(n)\rfloor}^{\lfloor nt_k(n)\rfloor}\right)\bigcap\left(\bigcap_{k\notin J}\left(O_{\lfloor nt_{k-1}(n)\rfloor}^{\lfloor nt_k(n)\rfloor}\right)^c\right)
\end{equation}
(with the convention $t_0(n)=0$). In order to lighten the notations, we will write for every $k\in [s]$ :
\[O_k(J)\overset{def}{=}\left\lbrace\begin{array}{ll}\phantom{(} O_{\lfloor nt_{k-1}(n)\rfloor}^{\lfloor nt_k(n)\rfloor}\phantom{)} & \text{if } k\in J, \\ \left(O_{\lfloor nt_{k-1}(n)\rfloor}^{\lfloor nt_k(n)\rfloor}\right)^c & \text{else.}\end{array}\right.\]
Thus with this new notation we have :
\begin{equation*}
B(J) = \bigcap_{k=1}^s O_k(J).
\end{equation*}

In fact, if we cut the integer interval $\llbracket 1,\lfloor nt_s(n)\rfloor\rrbracket$ into the $s$ intervals of the form $\llbracket \lfloor nt_{i-1}(n)\rfloor+1,\lfloor nt_i(n)\rfloor\rrbracket$, then the set $B(J)$ contains all coordinates which have been drawn an odd number of times on the $j$-th interval for all $j\in J$ and an even number of times on the $j$-th interval for all $j\notin J$.

For example, $B(\emptyset)$ is the set of coordinates which have been drawn an even number of times on each $\llbracket \lfloor nt_{i-1}(n)\rfloor+1,\lfloor nt_i(n)\rfloor\rrbracket$.

By construction, one can see that $\left\lbrace B(J):J\subset [s]\right\rbrace$ is a partition of $[K]$. Moreover we can see every set $O_0^{\lfloor nt_k(n)\rfloor}$ for $k\in [s]$ as a (disjoint) union of some $B(J)$, since the coordinates who have been drawn an odd number of times between $0$ and $\lfloor nt_k(n)\rfloor$ are the ones who have been drawn an odd number of times in an odd number of intervals preceding $\lfloor nt_k(n)\rfloor$, i.e. :
\begin{equation}\label{EquatUnionBJO}
O_0^{\lfloor nt_j(n)\rfloor}=\bigsqcup_{\substack{J\subset [s] \\ \vert J\cap [j]\vert \text{ odd}}}B(J).
\end{equation}

Combining \eqref{EquatUnionBJO} with the Proposition \ref{PropoChangNotat}, we get the following :
\begin{align}\label{EquatFinal}
X_{n,K}(\mathbf{t}_n)&\overset{\mathcal{D}}{=}\left(\sum_{i\in [K]}\left(\II_{i\notin O_0^{\lfloor nt_j(n)\rfloor}}-\II_{i\in O_0^{\lfloor nt_j(n)\rfloor}}\right)\xi_i\right)_{j\in [s]}\nonumber \\ 
&=\left(\sum_{J\subset [s]}(-1)^{\vert J\cap [j]\vert}\sum_{i\in B(J)}\xi_i\right)_{j\in [s]}.
\end{align}

\setcounter{truc}{\value{lemma}}

We now state a lemma (whose proof is in appendix), which is an ersatz of Donsker's Theorem with converging random time vectors :
\begin{lemma}\label{TheorDonskAleat}
Let $\left(\xi_i\right)_{i\in\NN}$ be i.i.d. Rademacher random variables and $\left(\mathbf{t}_K\right)_{K\geq 1}$ a sequence of random vectors in $\left(\RR_+\right)^k$ independent from $\left(\xi_i\right)_{i\in\NN}$. Define for all $K\geq 1$ and $\mathbf{s}=(s_1,\dots,s_k)\in\left(\RR_+\right)^k$ :
\[T_K(\mathbf{s})=\left(\frac{1}{\sqrt{K}}\sum_{i=1}^{\lfloor Ks_j\rfloor}\xi_i\right)_{j\in [k]}.\]

If there exists a deterministic $\mathbf{t}=(t_1,\dots,t_k)\in\left(\RR_+\right)^k$ such that $\mathbf{t}_K\overset{\PP}{\underset{K\to\infty}{\longrightarrow}}\mathbf{t}$, then :
\[T_K(\mathbf{t}_K)\overset{\mathcal{D}}{\underset{K\to\infty}{\longrightarrow}}\left(W_{t_j}\right)_{j\in [k]}\]
where $\left(W_s\right)_{s\in\RR_+}$ denotes a standard Brownian motion starting from $0$.
\end{lemma}

In order to apply the Lemma \ref{TheorDonskAleat}, we need to reorder the $\xi_i$ but also to prove that all the $\frac{\vert B(J)\vert}{K}$ converge to some constants.

\begin{lemma}\label{LemmeConveCardi}
Whatever the choices of $s\geq 1$, under the assumption \eqref{HypotTempsAleat} we get :
\[\forall \ J\subset [s], \ \frac{\vert B(J)\vert}{K}\overset{\LL^2}{\underset{\substack{n,K\to\infty \\ n/K\to\infty}}{\longrightarrow}}\frac{1}{2^s}.\]
\end{lemma}
\begin{proof}
We aim to compute the first two moments of $\frac{\vert B(J)\vert}{K}$ and show they converge to $2^{-s}$ and $0$ respectively. First we use the fact that the coordinates are exchangeable, namely :
\begin{equation}\label{EquatInterProba}
\forall \ i\in [K], \ \PP(i\in O_k(J))=\PP(1\in O_k(J))=\frac{1}{K}\sum_{j=1}^K\PP(j\in O_k(J)).
\end{equation}
Then, since the sets $(O_k(J))_{k\in [s]}$ are independent, we get :
\begin{align*}
\EE\left[\vert B(J)\vert\right] &= \EE\left[\sum_{i=1}^K\prod_{k=1}^s \II_{i\in O_k(J)}\right] = \sum_{i=1}^K\prod_{k=1}^s \PP(i\in O_k(J)) \\
&= \sum_{i=1}^K\prod_{k=1}^s \frac{1}{K}\sum_{j=1}^K \PP(j\in O_k(J)) = \frac{1}{K^{s-1}}\prod_{k=1}^s \EE\left[\vert O_k(J)\vert\right]
\end{align*}
In particular :
\begin{equation}\label{EquatEsperBJ}
\EE\left[ \frac{\vert B(J)\vert}{K}\right]=\prod_{k=1}^s \EE\left[\frac{\vert O_k(J)\vert}{K}\right].
\end{equation}
Next we use the Lemma \ref{LemmeConveEhren} (statement and proof some pages ahead) to get :
\[\EE\left[\frac{\vert O_k(J)\vert}{K}\right] \underset{\substack{n,K\to\infty \\ n/K\to\infty}}{\longrightarrow} \frac{1}{2}.\]
Then :
\[\EE\left[\frac{\vert B(J)\vert}{K}\right] \underset{\substack{n,K\to\infty \\ n/K\to\infty}}{\longrightarrow} \left(\frac{1}{2}\right)^s.\]

Now we compute the variance using \eqref{EquatInterProba} and \eqref{EquatEsperBJ} :
\begin{align*}
\VV\left[\frac{\vert B(J)\vert}{K}\right] &=\EE\left[\left(\frac{\vert B(J)\vert}{K}\right)^2\right] - \EE\left[\frac{\vert B(J)\vert}{K}\right]^2 \\
& = \EE\left[\frac{1}{K^2}\sum_{1\leq i,j\leq K}\II_{i\in B(J)}\II_{j\in B(J)}\right] - \left(\frac{1}{2^s}\right)^2 \\
& = \frac{1}{K^2}\sum_{1\leq i,j\leq K}\prod_{k=1}^s\PP(i\in O_k(J),j\in O_k(J)) \\
&= \frac{1}{K^2}\sum_{1\leq i\leq K}\prod_{k=1}^s\PP(1\in O_k(J)) \\
&\hspace{0.3cm} +\ \ \frac{1}{K^2}\sum_{\substack{1\leq i,j\leq K\\i\neq j}}\prod_{k=1}^s\PP(1\in O_k(J),2\in O_k(J)) \ \ - 4^{-s} \\
& = \frac{1}{K}\prod_{k=1}^s \EE\left[\frac{\vert O_k(J)\vert}{K}\right] \\
&\hspace{0.3cm} +\ \ \frac{K-1}{K}\prod_{k=1}^s\frac{1}{K(K-1)}\sum_{\substack{1\leq i,j\leq K\\i\neq j}}\EE\left[\II_{i\in O_k(J)}\II_{j\in O_k(J)}\right] \ \ -4^{-s} \\
& = \frac{1}{K}\EE\left[\frac{\vert B(J)\vert}{K}\right] \ \ +\ \ \frac{K-1}{K}\prod_{k=1}^s\EE\left[\frac{\vert O_k(J)\vert (\vert O_k(J)\vert -1)}{K(K-1)}\right]-\frac{1}{4^s}.
\end{align*}
The first term obviously tends to 0, so we just have to rewrite the second one in a more suitable way :
\begin{align*}
\frac{K-1}{K}\prod_{k=1}^s\EE\left[\frac{\vert O_k(J)\vert (\vert O_k(J)\vert -1)}{K(K-1)}\right]  =& \\
&\hspace{-2.4cm} \frac{K-1}{K}\prod_{k=1}^s \frac{K}{K-1}\left(\EE\left[\left(\frac{\vert O_k(J)\vert}{K}\right)^2\right]-\EE\left[\frac{\vert O_k(J)\vert}{K^2}\right]\right).
\end{align*}
Using the Lemma \ref{LemmeConveEhren}, we get for every $k\in [s]$ :
\begin{align*}
& \EE\left[\frac{\vert O_k(J)\vert}{K^2}\right] \underset{\substack{n,K\to\infty \\ n/K\to\infty}}{\longrightarrow} 0, \hspace{2cm} \ \EE\left[\left(\frac{\vert O_k(J)\vert}{K}\right)^2\right] \underset{\substack{n,K\to\infty \\ n/K\to\infty}}{\longrightarrow} \frac{1}{4}, \\
& \ \ \ \ \text{and finally }\ \VV\left[\frac{\vert B(J)\vert}{K}\right] \underset{\substack{n,K\to\infty \\ n/K\to\infty}}{\longrightarrow} 0+\frac{1}{4^s}-\frac{1}{4^s}=0.
\end{align*}
\end{proof}

\begin{lemma}\label{LemmeConveEhren}
For every $J\subset [s]$ and $k\in [s]$ we have the following convergence :
\[ \frac{\vert O_k(J)\vert}{K}\overset{\LL^2}{\underset{\substack{n,K\to\infty \\ n/K\to\infty}}{\longrightarrow}}\frac{1}{2}.\]
\end{lemma}
\begin{proof}
It's easy to check that $\left\vert O_i^j\right\vert\overset{\mathcal{D}}{=}E_K(j-i)$ where $\left(E_K(n)\right)_{n\in\NN}$ is a $K$-Ehrenfest's urn starting from 0.

Then we can use the moments of Ehrenfest's urn to get the convergence (see for instance \cite{Bar82} for computations of the first two moments of Ehrenfest's urn). Setting $\Delta_k(n) \overset{def}{=} \lfloor nt_k(n)\rfloor - \lfloor nt_{k-1}(n)\rfloor$ we have :
\begin{equation}\label{EquatEsperEhren}
\EE\left[O_{\lfloor nt_{k-1}(n)\rfloor}^{\lfloor nt_k(n)\rfloor}\right] = \frac{K}{2}\left(1-(1-\frac{2}{K})^{\Delta_k(n)}\right) = K-\EE\left[\left(O_{\lfloor nt_{k-1}(n)\rfloor}^{\lfloor nt_k(n)\rfloor}\right)^c\right]
\end{equation}
and then in whichever cases ($k\in J$ or $k\notin J$) we get :
\begin{align*}
\EE\left[\frac{\vert O_k(J)\vert}{K}\right] &= \frac{1}{2}\left(1\pm (1-\frac{2}{K})^{\Delta_k(n)}\right)= \frac{1}{2}\left(1\pm\exp\left(\Delta_k(n)\ln(1-\frac{2}{K})\right)\right) \\
&= \frac{1}{2}\left(1\pm\exp\left(\Delta_k(n)(-\frac{2}{K}+O(K^{-2}))\right)\right) \underset{\substack{n,K\to\infty \\ n/K\to\infty}}{\longrightarrow} \frac{1}{2}
\end{align*}
since we assumed $\frac{\Delta_k(n)}{K}\underset{n,K\to\infty }{\longrightarrow}+\infty$ in \eqref{HypotTempsAleat}.
And for the variance (which is the same in both cases) :
\begin{align*}
&\VV\left[\frac{\vert O_k(J)\vert}{K}\right] = \frac{1}{4}\left(\frac{1}{K}+\frac{K-1}{K}(1-\frac{4}{K})^{\Delta_k(n)} -(1-\frac{2}{K})^{2\Delta_k(n)}\right) \\
&= \frac{1}{4}\left(\frac{1}{K}+\frac{K-1}{K}\exp\left(-4\frac{\Delta_k(n)}{K}+O\left(\frac{\Delta_k(n)}{K^2}\right)\right)\right. \\
&\hspace{5cm}\left.-\exp\left(-4\frac{\Delta_k(n)}{K}+O\left(\frac{\Delta_k(n)}{K^2}\right)\right)\right) \\
&= \frac{1}{4K}\left(1-\exp\left(-4\frac{\Delta_k(n)}{K}+O\left(\frac{\Delta_k(n)}{K^2}\right)\right)\right) \underset{\substack{n,K\to\infty \\ n/K\to\infty}}{\longrightarrow} 0.
\end{align*}
\end{proof}

With the Lemmas \ref{TheorDonskAleat} and \ref{LemmeConveCardi}  we get the following convergence :

\begin{proposition}\label{PropoConveLFD}
\[\left( \frac{1}{\sqrt{K}}\sum_{i\in B(J)}\xi_i\right)_{J\subset [s]}\overset{\mathcal{D}}{\underset{\substack{n,K\to\infty\\ n/K\to\infty}}{\longrightarrow}}\left( G^{(s)}_J\right)_{J\subset [s]} \]
where $\left( G^{(s)}_J\right)_{J\subset [s]}$ denotes an i.i.d. collection of Gaussian random variables $\mathcal{N}(0,2^{-s})$.
\end{proposition}
\begin{proof}
Let's write $\lbrace J\subset [s]\rbrace = \lbrace J_1,J_2,\dots ,J_{2^s}\rbrace$ and set $S_k\overset{def}{=}\sum_{i=1}^k \vert B(J_i)\vert$. By reordering the $\xi_i$, we get :
\[\left(\frac{1}{\sqrt{K}}\sum_{i\in B(J)}\xi_i\right)_{J\subset [s]} \overset{\mathcal{D}}{=} \left(\frac{1}{\sqrt{K}}\sum_{i=S_{k-1}}^{S_k-1}\xi_i\right)_{k\in [2^s]}.\]

We can use the Lemma \ref{LemmeConveCardi} to show the following convergence :
\[\left(\frac{S_k}{K}\right)_{k\in [2^s]}\overset{\PP}{\underset{\substack{n,K\to\infty\\ n/K\to\infty}}{\longrightarrow}} \left(\frac{k}{2^s}\right)_{k\in [2^s]}.\]

Since $\left( S_k\right)_{k\in [2^s]}$ and $\left(\xi_i\right)_{i\in\NN}$ are independent, the Lemma \ref{TheorDonskAleat} states :
\[\left(\frac{1}{\sqrt{K}}\sum_{i=S_k}^{S_{k+1}}\xi_i\right)_{k\in [2^s]}\overset{\mathcal{D}}{\underset{\substack{n,K\to\infty\\ n/K\to\infty}}{\longrightarrow}} \left(W_{\frac{k+1}{2^s}}-W_{\frac{k}{2^s}}\right)_{k\in [2^s]}\]
where $\left( W_t\right)_{t\geq 0}$ is a standard Brownian motion starting from $0$. We can conclude thanks to the following :
\[\left(W_{\frac{k+1}{2^s}}-W_{\frac{k}{2^s}}\right)_{k\in [2^s]} \overset{\mathcal{D}}{=} \left( G^{(s)}_J\right)_{J\subset [s]}.\]
\end{proof}

We finally get the convergence of the rescaled finite dimensions laws of the processes $X_{n,K}$ under the fast regime :

\begin{theorem}\label{PropoConveInfin}Under the assumption \eqref{HypotTempsAleat} we have :
\[\frac{X_{n,K}(\mathbf{t}_n)}{\sqrt{K}} \overset{\mathcal{D}}{\underset{\substack{n,K\to\infty\\ n/K\to\infty}}{\longrightarrow}}\left( G_j\right)_{1\leq j\leq s}\]
where $\left( G_j\right)_{1\leq j\leq s}$ are i.i.d. Gaussian random variables $\mathcal{N}(0,1)$.
\end{theorem}
\begin{proof}
Using Proposition \ref{PropoConveLFD} we know that the limit is a centered Gaussian vector, so we just need to compute its covariance, namely $cov(G_i,G_j)$ for all $1\leq i,j\leq s$. In the case $i< j$, using \eqref{EquatFinal} with $\mathbf{t}'_n=(t_i(n),t_j(n))$ we get
\begin{align*}
&cov(G_i,G_j) = cov\left(\sum_{J\subset [2]}(-1)^{\vert J\cap [1]\vert}G_J^{(2)},\sum_{J\subset [2]}(-1)^{\vert J\cap [2]\vert}G_J^{(2)}\right) \\
&= cov(G_\emptyset^{(2)}-G_{\lbrace 1\rbrace}^{(2)}+G_{\lbrace 2\rbrace}^{(2)}-G_{\lbrace 1,2\rbrace}^{(2)},G_\emptyset^{(2)}-G_{\lbrace 1\rbrace}^{(2)}-G_{\lbrace 2\rbrace}^{(2)}+G_{\lbrace 1,2\rbrace}^{(2)}) \\
&= \VV[G_\emptyset^{(2)}]+\VV[G_{\lbrace 1\rbrace}^{(2)}]-\VV[G_{\lbrace 2\rbrace}^{(2)}]-\VV[G_{\lbrace 1,2\rbrace}^{(2)}] = 0.
\end{align*}
For the variance, just consider $s=1$ and we have for any $j$ :
\begin{align*}
\VV[G_j] = \VV\left[\sum_{J\subset [1]}(-1)^{\vert J\cap [1]\vert}G_J^{(1)}\right] = \VV[G_\emptyset^{(1)}-G_{\lbrace 1\rbrace}^{(1)}]  = \frac{1}{2}+\frac{1}{2} = 1.
\end{align*}
Then $\left(G_j\right)_{1\leq j \leq s}$ is a centered Gaussian vector whose covariance is given by :
\[cov(G_i,G_j)=\II_{i=j}\] which characterizes the standard normal random vector, i.e. a random vector whose marginals are i.i.d. $\mathcal{N}(0,1)$ random variables.
\end{proof}

\begin{corollary}
\[\left(\frac{X_{n,K}(t)}{\sqrt{K}}\right)_{t\geq 0}\overset{f.d.l.}{\underset{\substack{n,K\to\infty\\ n/K\to\infty}}{\longrightarrow}}\left( G_t\right)_{t\geq 0}\]
where $\left( G_t\right)_{t\geq 0}$ is an i.i.d. collection of Gaussian random variables $\mathcal{N}(0,1)$.
\end{corollary}
\begin{proof}
Any constant vector $\mathbf{t}=(t_1,\dots, t_s)$ such that $0\leq t_1<t_2<\dots < t_s$ fulfills the assumption \eqref{HypotTempsAleat}, and thus we can apply the previous Theorem.
\end{proof}

We could try to prove the results in the fast regime by dilating time in the intermediate one, but the fact that the dilation goes to infinity rises many technical issues. That is why we use another way to reach the asymptotic behavior in the fast regime.

\section{Slow regime}

Now we consider the slow regime, i.e. the case where $\frac{n}{K}\to 0$. This condition on the asymptotic of $\frac{n}{K}$ raises the following issue : if $Y_K(0)$ is distributed under the invariant law, $\VV[X_{n,K}(0)]=K$ while $\VV[X_{n,K}(t)-X_{n,K}(0)]$ is roughly $nt$. Then either we set $c_{n,K}=\sqrt{K}$ and the limit process is constant, or we consider $c_{n,K}=\sqrt{n}$ and the law of $\frac{X_{n,K}(0)}{c_{n,K}}$ diverges as $n$ and $K$ tend to infinity.

We will consider the second option, and then set $Z_{n,K}(t)=n^{-\frac{1}{2}}X_{n,K}(t)$. If $n^{-\frac{1}{2}}X_{n,K}(0)$ converges in law we can apply the Theorem 1, but it doesn't include the case when $Y_K(0)$ is distributed under the invariant law. 

We will then consider a wider class of initial laws, but in order to get rid of the "diverging initial value" problem we will focus on the increments of the processes, namely we set :
\[\Delta Z_{n,K}(t)\overset{def}{=}Z_{n,K}(t)-Z_{n,K}(0)=\frac{1}{\sqrt{n}}\left(X_{n,K}(t)-X_{n,K}(0)\right).\]

\begin{theorem}\label{PropoConveZero}
If the sequence $\mu_K$ is such that $\frac{\sqrt{n}}{K}X_{n,K}(0)\overset{\PP}{\underset{\substack{n,K\to\infty \\ n/K\to 0}}{\longrightarrow}}0$, then :
\[\left(\Delta Z_{n,K}(t)\right)_{t\geq 0}\overset{\mathcal{D}}{\underset{\substack{n,K\to\infty \\ n/K\to 0}}\longrightarrow}\left(2W_t\right)_{t\geq 0}\]
where $\left(W_t\right)_{t\geq 0}$ denotes a standard Brownian motion starting from $0$.
\end{theorem}

\begin{proof}
The Proposition \ref{EK} can no longer be used since the processes $\Delta Z_{n,K}$ are not Markovian. We need a more general result, namely the Theorem 4.1 from \cite{EK91}.

To apply this theorem, we need to find two sequences of random processes $A_n$ and $B_n$ such that $M_n(t)\overset{def}{=}\Delta Z_{n,K}(t)-B_n(t)$ and $M_n^2(t)-A_n(t)$ are $(\mathcal{F}_t^n)$-local martingales, with $\mathcal{F}_t^n=\sigma\left(\Delta Z_{n,K}(s),A_n(s),B_n(s):s\leq t\right)$.

\medskip

Let $\tilde{\mathcal{F}}_t^n=\sigma\left(X_{n,K}(s):s\leq t\right)$. If we note $\nabla X_{n,K}(i) \overset{def}{=}X_{n,K}\left(\frac{i+1}{n}\right)-X_{n,K}\left(\frac{i}{n}\right)$ we get for any integer $i\geq 0$ :
\begin{align*}
\EE\left[\left. \nabla X_{n,K}(i)\right\vert \tilde{\mathcal{F}}_{i/n}^n\right] &= \EE\left[ f_{K}(Y_{K}(i+1))\left.-f_K(Y_{K}(i)) \right\vert \tilde{\mathcal{F}}_{i/n}^n\right] \\
&=\EE\left[\left. \sum_{k=1}^{K} Y_{K}^{(k)}(i+1)-\sum_{k=1}^{K} Y_{K}^{(k)}(i)\right\vert Y_{K}(i)\right] \\
&=\EE\left[\left. -2Y_{K}^{(j)}(i)\right\vert Y_{K}(i)\right] = -2\sum_{k=1}^{K} \frac{Y_{K}^{(k)}(i)}{K} \\
&= -\frac{2}{K}X_{n,K}\left(\frac{i}{n}\right)
\end{align*}
where $j$ denotes the changing coordinate between $Y_{K}(i)$ and $Y_{K}(i+1)$. Similarly we set $\nabla Z_{n,K}(i) \overset{def}{=}\Delta Z_{n,K}\left(\frac{i+1}{n}\right)-\Delta Z_{n,K}\left(\frac{i}{n}\right)$ and then :
\begin{align*}
\EE\left[\left. \nabla  Z_{n,K}(i)\right\vert \tilde{\mathcal{F}}_{i/n}^n\right] &= \frac{1}{\sqrt{n}}\EE\left[\nabla  X_{n,K}(i)\left. \right\vert \tilde{\mathcal{F}}_{i/n}^n\right] = \frac{1}{\sqrt{n}}\times\frac{-2}{K}X_{n,K}\left(\frac{i}{n}\right) \\
&= -\frac{2}{K}\left(\Delta Z_{n,K}\left(\frac{i}{n}\right)+\frac{X_{n,K}(0)}{\sqrt{n}}\right). 
\end{align*}

So if we set :
\begin{align*}
B_n(t) &\overset{def}{=} \sum_{i=0}^{\lfloor nt\rfloor -1}\EE\left[\left. \nabla  Z_{n,K}(i)\right\vert \tilde{\mathcal{F}}_{i/n}^n\right] = -\frac{2}{K}\sum_{i=0}^{\lfloor nt\rfloor -1}\left(\Delta Z_{n,K}\left(\frac{i}{n}\right)+\frac{X_{n,K}(0)}{\sqrt{n}}\right) \\
A_n(t) &\overset{def}{=} \sum_{i=0}^{\lfloor nt\rfloor -1}\left(\EE\left[\left. \left(\nabla  Z_{n,K}(i)\right)^2\right\vert \tilde{\mathcal{F}}_{i/n}^n\right] -\EE\left[\left. \nabla Z_{n,K}(i)\right\vert \tilde{\mathcal{F}}_{i/n}^n\right]^2\right) \\
& \ = 4\frac{\lfloor nt\rfloor}{n}-\frac{4}{K^2}\sum_{i=0}^{\lfloor nt\rfloor -1}\left(\Delta Z_{n,K}\left(\frac{i}{n}\right)+\frac{X_{n,K}(0)}{\sqrt{n}}\right)^2
\end{align*}

then the processes $M_n(t)$ and $M_n^2(t)-A_n(t)$ will be $(\tilde{\mathcal{F}}_t^n)$-martingales, and then $(\mathcal{F}_t^N)$-martingales (since they are $(\mathcal{F}_t^N)$-adapted).

We now just have to check the technical requirements of the Theorem from \cite{EK91} :
\begin{proposition}
For all $T,r>0$, if $\tau_n^r=\inf\left\lbrace t:\vert \Delta Z_{n,K}(t)\vert\geq r\right\rbrace$ we have :
\begin{align*}
\EE\left[\sup_{0<t\leq T\wedge\tau_n^r}\left\vert\Delta Z_{n,K}(t)- \lim_{\varepsilon\to 0}\Delta Z_{n,K}(t-\varepsilon)\right\vert ^2\right]\underset{n\to\infty}{\longrightarrow}0, \\
\EE\left[\sup_{0<t\leq T\wedge\tau_n^r}\left\vert B_n(t)- \lim_{\varepsilon\to 0}B_n(t-\varepsilon)\right\vert ^2\right]\underset{n\to\infty}{\longrightarrow}0, \\
\EE\left[\sup_{0<t\leq T\wedge\tau_n^r}\left\vert A_n(t)- \lim_{\varepsilon\to 0}A_n(t-\varepsilon)\right\vert \right]\underset{n\to\infty}{\longrightarrow}0, \\
\sup_{t\leq T\wedge\tau_n^r}\vert B_n(t)\vert \overset{\PP}{\underset{n\to\infty}{\longrightarrow}}0 \ \text{and }\sup_{t\leq T\wedge\tau_n^r}\vert A_n(t)-4t\vert \overset{\PP}{\underset{n\to\infty}{\longrightarrow}}0.
\end{align*}
\end{proposition}
\begin{proof}
The first convergence is obvious since the jumps of $\Delta Z_{n,K}$ are bounded by $\frac{1}{\sqrt{n}}$.
Then we deal with the jumps of the compensator $B_n$ :
\begin{align*}
&\EE\left[\sup_{0<t\leq T\wedge\tau_n^r}\left\vert B_n(t)- \lim_{\varepsilon\to 0}B_n(t-\varepsilon)\right\vert ^2\right] \\
&\hspace{2cm} = \EE\left[\sup_{\frac{i}{n}\leq T\wedge\tau_n^r}\left\vert -\frac{2}{K}\left(\Delta Z_{n,K}\left(\frac{i}{n}\right)+\frac{X_{n,K}(0)}{\sqrt{n}}\right) \right\vert ^2\right] \\
&\hspace{2cm} \leq \EE\left[\sup_{\frac{i}{n}\leq T\wedge\tau_n^r}4\left(\left(\frac{\Delta Z_{n,K}(i/n)}{K}\right)^2 + \left( \frac{X_{n,K}(0)}{K\sqrt{n}} \right)^2\right)\right] \\
&\hspace{2cm} \leq 4\left(\left(\frac{r}{K}\right)^2 + \left( \frac{1}{\sqrt{n}} \right)^2\right)
\end{align*}
which obviously tends to 0. The same kind of computations applied to $A_n$ ensure the third convergence.

The trickiest assumptions in this proposition are the two last ones, namely~:
\[\sup_{t\leq T\wedge\tau_n^r}\vert B_n(t)\vert \overset{\PP}{\underset{n\to\infty}{\longrightarrow}}0 \text{ and }\sup_{t\leq T\wedge\tau_n^r}\vert A_n(t)-4t\vert \overset{\PP}{\underset{n\to\infty}{\longrightarrow}}0.\]
With a good use of the stopping times $T$ and $\tau_n^r$ we get the following bound :
\begin{align*}
\sup_{t\leq T\wedge\tau_n^r}\vert B_n(t)\vert &= \sup_{t\leq T\wedge\tau_n^r}\left\vert -\frac{2}{K}\left(\frac{\lfloor nt\rfloor X_{n,K}(0)}{\sqrt{n}}+\sum_{i=0}^{\lfloor nt\rfloor -1}\Delta Z_{n,K}\left(\frac{i}{n}\right)\right)\right\vert \\
&\leq \frac{2\lfloor nT\rfloor}{K\sqrt{n}}\vert X_{n,K}(0)\vert + \frac{2\lfloor nT\rfloor r}{K}
\end{align*}
where the second term tends to $0$ in this regime. Since we assumed :
\[\frac{\sqrt{n}}{K}X_{n,K}(0)\overset{\PP}{\underset{\substack{n,K\to\infty \\ n/K\to 0}}{\longrightarrow}}0\]
then the supremum of $\vert B_n\vert$ converges to $0$ in probability.

The last convergence may be proven exactly the same way.
\end{proof}

Then we fulfilled the assumptions of the Theorem 4.1 from \cite{EK91}, which prove that the processes $\Delta Z_{n,K}$ converge in distribution to a diffusion $(X_t)_{t\geq 0}$ such that $X_0=0$ a.s. and $dX_t=0dt + 2dW_t$. We easily conclude that :
\[\left(\Delta Z_{n,K}(t)\right)_{t\geq 0}\overset{\mathcal{D}}{\underset{n,K\to\infty}\longrightarrow}\left(2W_t\right)_{t\geq 0}.\]
\end{proof}

The Theorem \ref{PropoConveZero} describes correctly the increments of the processes $Z_{n,K}$ (especially in the stationary case), but does not help to approximate the processes themselves, since we lose the information of the initial value. In fact, if $Y_K(0)$ is uniform on $V_K$ the processes $Z_{n,K}$ tend to behave like a "stationary Brownian motion", which starts from its invariant measure aka the Lebesgue measure on $\RR$. Obviously this is no longer a stochastic process, this is why we cannot use a convergence in distribution, but instead we will prove a vague convergence of the finite-dimensional laws of our processes :

\begin{theorem}\label{TheorRegimLent}
If $\mu_K$ is the uniform law on $V_K$, then for any $0\leq t_1<\dots <t_s$ and for every smooth compactly supported function $\varphi\in\mathcal{C}_c^\infty (\RR^s,\RR)$ we have the following convergence :
\begin{align*}
&\sqrt{2\pi\frac{K}{n}}\times \EE\left[\varphi \left(Z_{n,K}(t_1),\Delta Z_{n,K}(t_2),\dots,\Delta Z_{n,K}(t_s) \right)\right] \\
&\hspace{5.5cm} \underset{n,K\to\infty}{\longrightarrow}\int_{\RR^s}\EE\left[\varphi\left(x,2W_{t_2},\dots,2W_{t_s}\right)\right]dx
\end{align*}
where $(W_t)_{t\geq 0}$ is a standard Brownian motion starting from $0$.
\end{theorem}

\begin{proof}
Since the processes $Z_{n,K}$ are temporally stationnary we can assume that $t_1=0$. We will write $\varphi(Z_{n,K})\overset{def}{=}\varphi \left(Z_{n,K}(0),\Delta Z_{n,K}(t_2),\dots,\Delta Z_{n,K}(t_s) \right)$ to lighten the notations, and consider the following consequence of Theorem \ref{RegimInter} from the intermediate regime :
\begin{corollary}
Let $(x_n)_{n\in\NN}$ be a deterministic sequence such that $x_n\underset{n\to\infty}{\longrightarrow}x$ for some $x\in\RR$. Then under the slow regime we get :
\begin{align*}
\EE\left[\left. \varphi(Z_{n,K})\right\vert Z_{n,K}(0) = x_n\right]\underset{n,K\to\infty}{\longrightarrow}\EE\left[\varphi\left(x, 2W_{t_2},\dots, 2W_{t_s}\right)\right].
\end{align*}
\end{corollary}
\begin{proof}
We apply the Theorem \ref{RegimInter} with $c_{n,K}=n^{-\frac{1}{2}}$, $\lambda = 0$ and $Z_{n,K}(0)\overset{a.s.}{=}x_n$. Then the limit process is just $(2W_t + x)_{t\geq 0}$ where $W$ is a standard Brownian motion starting from $0$.
\end{proof}
Then we can write the following :
\begin{align*}
&\sqrt{2\pi\frac{K}{n}}\times \EE\left[\varphi (Z_{n,K})\right]
= \EE\left[\sqrt{2\pi\frac{K}{n}}\times \EE\left[\left. \varphi (Z_{n,K})\right\vert Z_{n,K}(0) \right]\right] \\
&= \sum_{k\in\ZZ} \sqrt{2\pi\frac{K}{n}}\times \EE\left[\left. \varphi (Z_{n,K})\right\vert Z_{n,K}(0)=\frac{k}{\sqrt{n}} \right]\times\PP\left(Z_{n,K}(0)=\frac{k}{\sqrt{n}}\right)\\
&= \int_\RR \sqrt{2\pi K}\times \EE\left[\left. \varphi (Z_{n,K})\right\vert Z_{n,K}(0)=\frac{\lfloor\sqrt{n}x\rfloor}{\sqrt{n}} \right]\times\PP\left(Z_{n,K}(0)=\frac{\lfloor\sqrt{n}x\rfloor}{\sqrt{n}}\right)dx.
\end{align*}
We will need the following result (which is a consequence of De Moivre-Laplace Theorem, see for instance \cite{GK68}) to replace the probability in the previous line by a "Gaussian equivalent" :
\begin{theorem}\label{DLLT}
If $(\xi_i)_{i\in\NN}$ is an i.i.d. sequence of Rademacher random variables, then for every $N\geq 1$ :
\begin{equation*}
\sup_{m\in\ZZ}\sqrt{N}\times\left\vert\PP\left(\sum_{i=1}^N \xi_i = m\right) - \frac{2}{\sqrt{2\pi N}}e^{-\frac{m^2}{2N}}\II_{m\equiv N[2]}\right\vert\underset{N\to\infty}{\longrightarrow} 0
\end{equation*}
where $x\equiv y[2]$ means that the integers $x$ and $y$ have same parity.
\end{theorem}
In our case we would get :
\begin{equation*}
\sup_{x\in\RR}\sqrt{K}\times\left\vert\PP\left(Z_{n,K}(0) = \frac{\lfloor \sqrt{n}x\rfloor}{\sqrt{n}}\right) - \frac{2}{\sqrt{2\pi K}}e^{-\frac{\lfloor \sqrt{n}x\rfloor^2}{2K}}\II_{\lfloor \sqrt{n}x\rfloor\equiv K[2]}\right\vert\underset{n,K\to\infty}{\longrightarrow} 0.
\end{equation*}
Using the fact that $\varphi$ is compactly supported, there exists some compact set $C_0$ such that $x\notin C_0 \Longrightarrow  \EE\left[\left. \varphi (Z_{n,K})\right\vert Z_{n,K}(0)=\frac{\lfloor\sqrt{n}x\rfloor}{\sqrt{n}} \right]=0$. Then we have :
\begin{align*}
&\left\vert \int_\RR \sqrt{2\pi K}\times \EE\left[\left. \varphi (Z_{n,K})\right\vert Z_{n,K}(0)=\frac{\lfloor\sqrt{n}x\rfloor}{\sqrt{n}} \right]\times\PP\left(Z_{n,K}(0)=\frac{\lfloor\sqrt{n}x\rfloor}{\sqrt{n}}\right)dx \right. \\
&\hspace{2cm}\left. - \int_\RR \EE\left[\left. \varphi (Z_{n,K})\right\vert Z_{n,K}(0)=\frac{\lfloor\sqrt{n}x\rfloor}{\sqrt{n}} \right]2e^{-\frac{\lfloor \sqrt{n}x\rfloor^2}{2K}}\II_{\lfloor \sqrt{n}x\rfloor\equiv K[2]}dx\right\vert \\
&\hspace{0.5cm} \leq \vert\vert\varphi\vert\vert_\infty \int_\RR\sqrt{2\pi K}\left\vert\PP\left(Z_{n,K}(0) = \frac{\lfloor \sqrt{n}x\rfloor}{\sqrt{n}}\right) \right. \\
&\hspace{5cm} \left. - \frac{2}{\sqrt{2\pi K}}e^{-\frac{\lfloor \sqrt{n}x\rfloor^2}{2K}}\II_{\lfloor \sqrt{n}x\rfloor\equiv K[2]}\right\vert\II_{x\in C_0}dx
\end{align*}
which tends to $0$ as $n$ and $K$ tend to infinity.

The next step is to get rid of the "same parity indicator", but the reader can easily check that the error term we get by turning $\II_{\lfloor nx\rfloor\equiv K[2]}$ into $\frac{1}{2}$ will also vanish as $n$ and $K$ go to infinity (thanks to the uniform continuity of $\varphi$ and the compactness of $C_0$).

Then :
\begin{align*}
&\left\vert \sqrt{2\pi\frac{K}{n}}\times \EE\left[\varphi (Z_{n,K})\right] - \int_\RR \EE\left[\left. \varphi (Z_{n,K})\right\vert Z_{n,K}(0)=\frac{\lfloor\sqrt{n}x\rfloor}{\sqrt{n}} \right]e^{-\frac{\lfloor \sqrt{n}x\rfloor^2}{2K}}dx\right\vert \\
&\hspace{1cm}\underset{n,K\to\infty}{\longrightarrow} 0. 
\end{align*}

The result of Theorem \ref{TheorRegimLent} follows from dominated convergence since in the slow regime :
\begin{align*}
&\EE\left[\left. \varphi (Z_{n,K})\right\vert Z_{n,K}(0)=\frac{\lfloor\sqrt{n}x\rfloor}{\sqrt{n}} \right] \overset{\forall  x\in\RR}{\underset{n,K\to\infty}{\longrightarrow}} \EE\left[\varphi\left(x, 2W_{t_2},\dots, 2W_{t_s}\right)\right], \\
&\hspace{1cm}e^{-\frac{\lfloor \sqrt{n}x\rfloor^2}{2K}}\overset{\forall  x\in\RR}{\underset{n,K\to\infty}{\longrightarrow}} 1 \hspace{0.5cm} \text{and}\hspace{0.5cm} \int_\RR \vert\vert\varphi\vert\vert_\infty\II_{x\in C_0}dx <+\infty.
\end{align*}
\end{proof}

\appendix

\section*{Appendix 1 : Proof of Lemma \ref{TheorDonskAleat}}

Recall the Lemma \ref{TheorDonskAleat} which comes from the "Fast regime" section :

\begin{Lemma2}
Let $\left(\xi_i\right)_{i\in\NN}$ be i.i.d. Rademacher random variables and $\left(\mathbf{t}_K\right)_{K\geq 1}$ a sequence of random vectors in $\left(\RR_+\right)^k$ independent from $\left(\xi_i\right)_{i\in\NN}$. Define for all $K\geq 1$ and $\mathbf{s}=(s_1,\dots,s_k)\in\left(\RR_+\right)^k$ :
\[T_K(\mathbf{s})=\left(\frac{1}{\sqrt{K}}\sum_{i=1}^{\lfloor Ks_j\rfloor}\xi_i\right)_{j\in [k]}.\]

If there exists a deterministic $\mathbf{t}=(t_1,\dots,t_k)\in\left(\RR_+\right)^k$ such that $\mathbf{t}_K\overset{\PP}{\underset{K\to\infty}{\longrightarrow}}\mathbf{t}$, then :
\[T_K(\mathbf{t}_K)\overset{\mathcal{D}}{\underset{K\to\infty}{\longrightarrow}}\left(W_{t_j}\right)_{j\in [k]}\]
where $\left(W_s\right)_{s\in\RR_+}$ denotes a standard Brownian motion starting from $0$.
\end{Lemma2}

For all $\mathbf{x}=(x_1,\dots,x_k)\in\RR^k$ we will consider the following norm :
\[\vert\vert\mathbf{x}\vert\vert = \sup_{i\in [k]} \vert x_i\vert.\]
For the sake of simplicity we will set $\RR_+^k\overset{def}{=}\left(\RR_+\right)^k$, and for every $\RR^k$-valued process $\left(X_t\right)_{t\geq 0}$ and $\mathbf{t}=(t_1,\dots,t_k)\in\RR_+^k$ we will note $X_\mathbf{t}\overset{def}{=}\left(X_{t_1},\dots,X_{t_k}\right)$.

In order to prove the Lemma \ref{TheorDonskAleat}, we will use several times the following result :
\begin{lemma}
Let $\left(W_t\right)_{t\in\RR_+}$ be a standard Brownian motion starting from $0$. Then for all $\varepsilon>0$ and $d>0$ there exists $\delta>0$ such that :
\[\vert\vert \mathbf{x}-\mathbf{y}\vert\vert \leq \delta \ \Rightarrow \ \PP(\vert\vert W_\mathbf{x}-W_\mathbf{y}\vert\vert > d)< \varepsilon. \] \label{LemmeBrown}
\end{lemma}
\begin{proof}
\begin{align*}
\PP(\vert\vert W_\mathbf{x}-W_\mathbf{y}\vert \vert > d) &= \PP\left(\sup_{1\leq i\leq k} \vert  W_{x_i}-W_{y_i} \vert > d\right) \\
&\leq \sum_{i=1}^k \PP(\vert  W_{x_i}-W_{y_i} \vert > d) \\
&= \sum_{i=1}^k \PP(\vert  N_i \vert > d)
\end{align*}
where $N_i$ is a $\mathcal{N}(0,\vert x_i-y_i\vert)$ random variable. One can easily find a suitable $\delta>0$ such that~:
\[\vert x_i-y_i\vert \leq \delta \ \Rightarrow \ \PP(\vert N_i\vert > d)< \frac{\varepsilon}{k} \]
and get the result using the fact that $\vert\vert \mathbf{x}-\mathbf{y}\vert\vert\leq \delta \Rightarrow \vert x_i-y_i\vert \leq \delta$ for all $i\in [k]$.
\end{proof}

Let's get back to the main proof. Assuming the premises of Lemma \ref{TheorDonskAleat}, we want to prove that for every function $f:\RR_+^k\rightarrow \RR$ bounded and uniformly continuous we have :
\[\EE[f(T_K(\mathbf{t}_K))]\underset{K\to\infty}{\longrightarrow}\EE[f(W_\mathbf{t})].\]

Let $\varepsilon >0$. For all $\delta >  0$ we have :
\begin{align*}
\vert \EE[f(T_K(\mathbf{t}_K))] - \EE[f(W_\mathbf{t})] \vert &\leq \vert\EE[(f(T_K(\mathbf{t}_K)) - f(W_\mathbf{t}))\II_{\vert\vert \mathbf{t}_K-\mathbf{t}\vert\vert >\delta}]\vert\\
& \ \ + \vert\EE[(f(T_K(\mathbf{t}_K)) - f(W_\mathbf{t}))\II_{\vert\vert \mathbf{t}_K-\mathbf{t}\vert\vert\leq\delta}]\vert \\
&\leq 2\vert \vert f\vert\vert_\infty\PP(\vert\vert \mathbf{t}_K-\mathbf{t}\vert\vert >\delta) \\
& \ \ + \vert\EE[(f(T_K(\mathbf{t}_K)) - f(W_\mathbf{t}))\II_{\vert\vert \mathbf{t}_K-\mathbf{t}\vert\vert\leq\delta}]\vert .
\end{align*}

Since $\mathbf{t}_K\overset{\PP}{\underset{K\to\infty}{\longrightarrow}}\mathbf{t}$ then $\forall \ \delta >0$ there exists $N_1(\varepsilon,\delta)\in\NN$ such that for every $K\geq N_1(\varepsilon,\delta)$ we have :
\begin{equation}\label{DominProba1}
 \PP(\vert\vert \mathbf{t}_K-\mathbf{t}\vert\vert >\delta) \leq \frac{\varepsilon}{8\vert\vert f\vert\vert_\infty}.
 \end{equation}

Now we split the other term in two parts (which will be dominated separately) :
\begin{align*}\vert\EE[(f(T_K(\mathbf{t}_K)) - f(W_\mathbf{t}))\II_{\vert\vert \mathbf{t}_K-\mathbf{t}\vert\vert\leq\delta}]\vert \leq &\vert\EE[(f(T_K(\mathbf{t}_K)) - f(W_{\mathbf{t}_K}))\II_{\vert\vert \mathbf{t}_K-\mathbf{t}\vert\vert\leq\delta}]\vert \\
&+ \vert\EE[(f(W_{\mathbf{t}_K}) - f(W_\mathbf{t}))\II_{\vert\vert \mathbf{t}_K-\mathbf{t}\vert\vert\leq\delta}]\vert.
\end{align*}

Let's define the functions $\varphi_K$ and $\varphi$ by :
\[\varphi_K(\mathbf{s})=\EE[f(T_K(\mathbf{s}))] \text{ and } \varphi(\mathbf{s})=\EE[f(W_\mathbf{s})] \  \text{for } \mathbf{s}\in\RR_+^k .\]

\begin{proposition}
The functions $\left(\varphi_K\right)_{K\geq 1}$ converge uniformly on every compact of $\RR_+^k$ to the function $K$.
\end{proposition}
\begin{proof}
We already know via the Donsker's theorem that the sequence $\left(\varphi_K\right)_{K\geq 1}$ converges pointwise to $\varphi$ (see for instance \cite{KS91}).

Let $S\subset \RR_+^k$ compact, for all $\delta>0$ there exists a finite subset $\mathcal{M}\subset S$ such that :
\[\forall \ \mathbf{x}\in S \ \exists \ \mathbf{y}\in\mathcal{M} : \vert\vert \mathbf{x}-\mathbf{y}\vert\vert\leq \delta.\]

Let $\varepsilon >0$ and $\mathbf{t}\in S$, and choose $\mathbf{s}\in\mathcal{M}$ such that $\vert\vert \mathbf{t}-\mathbf{s}\vert\vert\leq \delta$. We then get :
\begin{equation}\label{EquatDecouTrois}
\vert\varphi_K(\mathbf{t})-\varphi(\mathbf{t})\vert \leq \vert\varphi_K(\mathbf{t})-\varphi_K(\mathbf{s})\vert + \vert\varphi_K(\mathbf{s})-\varphi(\mathbf{s})\vert + \vert\varphi(\mathbf{s})-\varphi(\mathbf{t})\vert.
\end{equation}

Let's dominate the first term. The function $f$ being uniformly continuous, there exists $d >0$ such that $\vert\vert \mathbf{x}-\mathbf{y}\vert\vert \leq d \Rightarrow \vert f(\mathbf{x})-f(\mathbf{y})\vert < \varepsilon/6$. Then we get :
\begin{align*}
\vert\varphi_K(\mathbf{t})-\varphi_K(\mathbf{s})\vert &\leq \EE[\vert f(T_K(\mathbf{t}))-f(T_K(\mathbf{s}))\vert] \phantom{qofihsgofg} \\
&  = \EE[\vert f(T_K(\mathbf{t}))-f(T_K(\mathbf{s}))\vert\II_{\vert\vert T_K(\mathbf{t})-T_K(\mathbf{s})\vert\vert \leq d}] \\
& \hspace{1cm}  + \EE[\vert f(T_K(\mathbf{t}))-f(T_K(\mathbf{s}))\vert\II_{\vert\vert T_K(\mathbf{t})-T_K(\mathbf{s})\vert\vert > d}] \\
& \hspace{-0.5cm} \leq \frac{\epsilon}{6}\PP(\vert\vert T_K(\mathbf{t})-T_K(\mathbf{s})\vert\vert \leq d) +2\vert\vert f \vert\vert_\infty\PP(\vert\vert T_K(\mathbf{t})-T_K(\mathbf{s})\vert\vert > d).
\end{align*}

But we can see in the proof of the Theorem 4.20 (p70) from \cite{KS91} that $\forall \  c >0$ et $\forall \ D>0$~:
\begin{equation*}
\lim_{\delta \to 0} \sup_{n\geq 1} \PP\left(\max_{\substack{\vert x-y\vert \leq \delta \\ 0\leq x,y \leq D}} \left\vert T_K(x)-T_K(y)\right\vert >c\right) =0. \label{KaratShrev}
\end{equation*}

So there exists some $\delta_1 (\varepsilon,d)>0$ such that $\forall\ \delta\leq \delta_1 (\varepsilon,d)$ :
\[ \sup_{n\geq 1} \PP\left(\max_{\substack{\vert x-y\vert \leq \delta \\ 0\leq x,y \leq D}} \left\vert T_K(x)-T_K(y)\right\vert >d\right)\leq \frac{\varepsilon}{12\vert\vert f\vert\vert_\infty k}\]
and then, setting $D(S)\overset{def}{=}\underset{\mathbf{z}\in S}{\sup}\vert\vert \mathbf{z}\vert\vert$, if we choose $\delta$ smaller than $\delta_1 (\varepsilon,d)$ we have $\forall \ K\geq 1$ :
\begin{align*}
\PP(\vert\vert T_K(\mathbf{t})-T_K(\mathbf{s})\vert\vert > d) &\leq k \times \PP\left(\max_{\substack{\vert x-y\vert \leq \delta \\ 0\leq x,y \leq D(S)}} \left\vert T_K(x)-T_K(y)\right\vert >d\right) \\
&\leq \frac{\varepsilon k}{12\vert\vert f\vert\vert_\infty k}=\frac{\varepsilon}{12\vert\vert f\vert\vert_\infty }.
\end{align*}

We finally get :
\begin{align*}
\vert\varphi_K(\mathbf{t})-\varphi_K(\mathbf{s})\vert \leq {} &  \frac{\epsilon}{6}\PP(\vert\vert T_K(\mathbf{t})-T_K(\mathbf{s})\vert\vert \leq d) \\
& +2\vert\vert f \vert\vert_\infty\PP(\vert\vert T_K(\mathbf{t})-T_K(\mathbf{s})\vert\vert > d) \\
\leq {} & \frac{\epsilon}{6} + 2\vert\vert f \vert\vert_\infty \frac{\varepsilon }{12\vert\vert f\vert\vert_\infty } =\frac{\varepsilon}{3}.
\end{align*}

Let's deal with the second term in \eqref{EquatDecouTrois}, namely $\vert\varphi_K(\mathbf{s})-\varphi(\mathbf{s})\vert$. Since $\varphi_K$ converges pointwise to $\varphi$, for every $\mathbf{x}\in\RR_+^k$ there exists $N_\mathbf{x}(\varepsilon)$ such that :
\[K\geq N_\mathbf{x}(\varepsilon) \Longrightarrow \vert\varphi_K(\mathbf{x})-\varphi(\mathbf{x})\vert \leq \frac{\varepsilon}{3}.\]
We just have to take $K$ greater than $N(\varepsilon,\delta)\overset{def}{=}\underset{\mathbf{x}\in \mathcal{M}}{\max}\ N_\mathbf{x}(\varepsilon)$ to get :
\[\vert\varphi_K(\mathbf{s})-\varphi(\mathbf{s})\vert \leq \frac{\varepsilon}{3}.\]

For the third term in \eqref{EquatDecouTrois}, we use again the uniform continuity of $f$ to get the following :
\begin{align*}
\vert\varphi(\mathbf{s})-\varphi(\mathbf{t})\vert &\leq \EE[\vert f(W_\mathbf{s})-f(W_\mathbf{t})\vert\II_{\vert\vert W_\mathbf{s}-W_\mathbf{t}\vert \vert \leq d}] \\
&\hspace{1cm} + \EE[\vert f(W_\mathbf{s})-f(W_\mathbf{t})\vert\II_{\vert\vert W_\mathbf{s}-W_\mathbf{t}\vert \vert > d}]\\
&\leq \frac{\varepsilon}{6}\PP(\vert\vert W_\mathbf{s}-W_\mathbf{t}\vert \vert \leq d) + 2\vert\vert f\vert\vert_\infty \PP(\vert\vert W_\mathbf{s}-W_\mathbf{t}\vert \vert > d).
\end{align*}

Using the Lemma \ref{LemmeBrown}, we find that $\exists \ \delta_2(\varepsilon,d)$ such that :
\[\delta\leq\delta_2(\varepsilon,d) \Longrightarrow \PP(\vert\vert W_\mathbf{s}-W_\mathbf{t}\vert \vert > d) <\frac{\varepsilon}{12\vert\vert f\vert\vert_\infty}.\]

We then have :
\begin{align*}
\vert\varphi(\mathbf{s})-\varphi(\mathbf{t})\vert &\leq  \frac{\varepsilon}{6}\PP(\vert\vert W_\mathbf{s}-W_\mathbf{t}\vert \vert \leq d) + 2\vert\vert f\vert\vert_\infty \PP(\vert\vert W_\mathbf{s}-W_\mathbf{t}\vert \vert > d)\\
&\leq \frac{\varepsilon}{6} + 2\vert\vert f\vert\vert_\infty  \frac{\varepsilon}{12\vert\vert f \vert\vert_\infty } =\frac{\varepsilon}{3}.
\end{align*}

Grouping all these results in \eqref{EquatDecouTrois}, we get that if $\delta\leq \min(\delta_1(\varepsilon,d),\delta(\varepsilon,d))$ then $\forall \ K\geq N(\varepsilon,\delta)$ we have :
\begin{equation*}
\sup_{\mathbf{t}\in S}\vert\varphi_K(\mathbf{t})-\varphi(\mathbf{t})\vert \leq \frac{\varepsilon}{3}+\frac{\varepsilon}{3}+\frac{\varepsilon}{3}=\varepsilon.
\end{equation*}
\end{proof}

Using the previous Proposition we get that for every $S$ compact subset of $\RR_+^k$ there exists $N(\varepsilon,S)$ such that $\forall \ \mathbf{s}\in S$ and $\forall \ K\geq N(\varepsilon,S)$ :
\[\vert \varphi_K(\mathbf{s})-\varphi(\mathbf{s})\vert \leq \frac{\varepsilon}{4}.\]
Using the independence of the $(\mathbf{t}_K)_K$ and the $(\xi_i)_i$ we can write that :
\begin{align*}
\EE[f(T_K(\mathbf{t}_K))] &= \EE[\EE[f(T_K(\mathbf{t}_K))\vert \mathbf{t}_K]] = \EE[\varphi_K(\mathbf{t}_K)].
\end{align*}

Then, since the ball of radius $\delta$ centered on $\mathbf{t}$ is compact, there exists an integer $N_2(\varepsilon,\delta)\overset{def}{=}N(\varepsilon,B(\mathbf{t},\delta))$ such that for all $K\geq N_2(\varepsilon,\delta)$ we have :
\begin{align}\label{DominEspe2}
\vert\EE [(f(T_K(\mathbf{t}_K)) - f(W_{\mathbf{t}_K}))\II_{\vert\vert \mathbf{t}_K-\mathbf{t}\vert\vert\leq\delta}]\vert &= \vert\EE[\EE[(f(T_K(\mathbf{t}_K))\nonumber \\ \nonumber
&\hspace{1cm}- f(W_{\mathbf{t}_K}))\vert \mathbf{t}_K]\II_{\vert\vert \mathbf{t}_K-\mathbf{t}\vert\vert\leq\delta}]\vert \\ \nonumber
&=\vert\EE[(\varphi_K(\mathbf{t}_K)-\varphi(\mathbf{t}_K))\II_{\vert\vert \mathbf{t}_K-\mathbf{t}\vert\vert\leq\delta}]\vert \\ \nonumber
&\leq \EE[\vert \varphi_K(\mathbf{t}_K)-\varphi(\mathbf{t}_K)\vert\II_{\mathbf{t}_K\in B(\mathbf{t},\delta)}] \\ 
&\leq \frac{\varepsilon}{4} \PP(\mathbf{t}_K\in B(\mathbf{t},\delta)) \leq \frac{\varepsilon}{4}.
\end{align}

Now we have to dominate $\vert\EE[(f(W_{\mathbf{t}_K}) - f(W_\mathbf{t}))\II_{\vert\vert \mathbf{t}_K-\mathbf{t}\vert\vert\leq\delta}]\vert $. Using the uniform continuity of $f$ we know there exists some $d>0$ such that :
\[\vert\vert \mathbf{x}-\mathbf{y}\vert\vert \leq d \Rightarrow \vert f(\mathbf{x})-f(\mathbf{y})\vert \leq \frac{\varepsilon}{4}.\]
We then have :
\begin{align}\label{DominEspe3}
\vert\EE[(f(W_{\mathbf{t}_K}) - f(W_\mathbf{t}))\II_{\vert\vert \mathbf{t}_K-\mathbf{t}\vert\vert\leq\delta}]\vert & \nonumber \\ \nonumber
&\hspace{-2cm} \leq  \EE[\vert f(W_{\mathbf{t}_K}) - f(W_\mathbf{t})\vert\II_{\vert\vert \mathbf{t}_K-\mathbf{t}\vert\vert\leq\delta}\II_{\vert\vert W_{\mathbf{t}_K}-W_\mathbf{t}\vert\vert \leq d}] \\ \nonumber
&\hspace{-1cm} + \EE[\vert f(W_{\mathbf{t}_K}) - f(W_\mathbf{t})\vert\II_{\vert\vert \mathbf{t}_K-\mathbf{t}\vert\vert\leq\delta}\II_{\vert\vert W_{\mathbf{t}_K}-W_\mathbf{t}\vert\vert > d}] \\ \nonumber
&\hspace{-2cm}\leq \frac{\varepsilon}{4}\PP(\vert\vert \mathbf{t}_K-\mathbf{t}\vert\vert\leq\delta ,\vert\vert W_{\mathbf{t}_K}-W_\mathbf{t}\vert\vert \leq d) \\ 
&\hspace{-1cm} + 2\vert\vert f \vert\vert_\infty  \PP(\vert\vert \mathbf{t}_K-\mathbf{t}\vert\vert\leq\delta ,\vert\vert W_{\mathbf{t}_K}-W_\mathbf{t}\vert\vert > d).
\end{align}

Using the Lemma \ref{LemmeBrown}, for all $d>0$ there exists $\delta(\varepsilon, d) >0$ small enough such that the following holds :
\begin{equation}\label{DominProba4}
\PP(\vert\vert \mathbf{t}_K-\mathbf{t}\vert\vert\leq\delta ,\vert\vert W_{\mathbf{t}_K}-W_\mathbf{t}\vert\vert > d)\leq \frac{\varepsilon}{8\vert\vert f \vert\vert_\infty}.
\end{equation}

In a nutshell if we sum up all the previous step, we find that for every bounded continuous function $f$ and $\forall \ \varepsilon >0$, we can choose $d\leq d(\varepsilon)$ such that \eqref{DominEspe3} holds, $\delta \leq \delta(\varepsilon,d)$ to have \eqref{DominProba4}, and then $K\geq \max(N_1(\varepsilon,\delta),N_2(\varepsilon,\delta)$ to get \eqref{DominProba1} and \eqref{DominEspe2}, which finally yields :
\begin{align*}
\vert \EE[f(T_K(\mathbf{t}_K))] - &\EE[f(W_\mathbf{t})] \vert \leq  \EE[\vert f(T_K(\mathbf{t}_K)) - f(W_\mathbf{t})\vert\II_{\vert\vert \mathbf{t}_K-\mathbf{t}\vert\vert >\delta}] \\
& + \EE[\vert f(T_K(\mathbf{t}_K)) - f(W_{\mathbf{t}_K})\vert\II_{\vert\vert \mathbf{t}_K-\mathbf{t}\vert\vert\leq\delta}] \\
& + \EE[\vert f(W_{\mathbf{t}_K}) - f(W_\mathbf{t})\vert\II_{\vert\vert \mathbf{t}_K-\mathbf{t}\vert\vert\leq\delta}\II_{\vert\vert W_{\mathbf{t}_K}-W_\mathbf{t}\vert\vert \leq d}] \\
& + \EE[\vert f(W_{\mathbf{t}_K}) - f(W_\mathbf{t})\vert\II_{\vert\vert \mathbf{t}_K-\mathbf{t}\vert\vert\leq\delta}\II_{\vert\vert W_{\mathbf{t}_K}-W_\mathbf{t}\vert\vert > d}] \\
&\leq  2\vert\vert f \vert\vert_\infty \frac{\varepsilon}{8\vert\vert f\vert\vert_\infty} + \frac{\varepsilon}{4} + \frac{\varepsilon}{4} + 2\vert\vert f \vert\vert_\infty \frac{\varepsilon}{8\vert\vert f\vert\vert_\infty} = \varepsilon
\end{align*}
and thus $T_K(\mathbf{t_K})\overset{\mathcal{D}}{\underset{K\to\infty}{\longrightarrow}}W_\mathbf{t}$.

\begin{acknowledgements}

I greatly thank my Ph.D advisor Serge Cohen for his constant help, and also Laurent Miclo, Charles Bordenave and Philippe Berthet for fruitful discussions and advices.

\end{acknowledgements}






\end{document}